\documentclass[12pt]{article}
\usepackage{amssymb}
\usepackage{amsmath}
\usepackage{amsthm}
\usepackage[showonlyrefs]{mathtools}
\mathtoolsset{showonlyrefs}
\usepackage[breaklinks=true,colorlinks, linkcolor=blue, anchorcolor=blue, citecolor=blue]{hyperref}
\numberwithin{equation}{section}
\def\re{\operatorname{Re}}
\def\im{\operatorname{Im}}
\def\dim{\operatorname{dim}}
\def\diam{\operatorname{diam}}

\def\sing{\operatorname{sing}}

\def\res{\operatorname{res}}

\def\N{\mathbb{N}}
\def\C{\mathbb{C}}
\def\Z{\mathbb{Z}}
\def\CC{\widehat{\mathbb C}}
\def\R{\mathbb{R}}
\def\Speiser{\mathcal{S}}
\def\B{\mathcal{B}}
\def\O{\mathcal{O}}
\newtheorem{theorem}{Theorem}[section]
\newtheorem{theorema}{Theorem}

\newtheorem{lemma}[theorem]{Lemma}
\newtheorem{corollary}[theorem]{Corollary}
\theoremstyle{remark}

\theoremstyle{definition}
\newtheorem*{ack}{Acknowledgment}
\begin{document}
\title{The Hausdorff dimension of escaping sets of meromorphic functions in the Speiser class}
\author{Walter Bergweiler and Weiwei Cui}
\date{}
\maketitle
\begin{abstract}
Bergweiler and Kotus gave sharp 
upper bounds for the Hausdorff dimension of the escaping set of a
meromorphic function in the Eremenko-Lyubich class, in terms of   the order of the 
function and the maximal multiplicity of the poles.
We show that these bounds are also sharp in the Speiser class. 
We apply this method also to construct meromorphic functions in the Speiser class with
preassigned dimensions of the Julia set and the escaping set.
\end{abstract}

\section{Introduction and main results} \label{intro}

Let $f\colon \C\to\CC:=\C\cup\{\infty\}$ be a transcendental meromorphic function.
The Fatou set $F(f)$ of $f$ is the set of all points $z\in\C$ which have a
neighborhood where the iterates $f^n$ of $f$ are defined and form a normal family.
The complement $J(f)=\C\setminus F(f)$ is called the Julia set.
These sets play a fundamental role in complex dynamics; see~\cite{Bergweiler1993} for
an introduction to the dynamics of transcendental meromorphic functions.
Besides $F(f)$ and $J(f)$, a key role in transcendental dynamics is played by the
escaping set $I(f)$ of~$f$. It consists of the points $z\in\C$ for which the iterates
$f^{n}(z)$ tend to $\infty$ as $n\to\infty$. It is known that $I(f)\neq\emptyset$ and
$J(f)=\partial I(f)$; see~\cite{Eremenko1989,Dominguez1998}.

Let $\sing(f^{-1})$ be the set of singularities of the inverse function of $f$;
that is, the set of critical and asymptotic values of~$f$.
The Eremenko-Lyubich class $\B$ consists of the transcendental meromorphic
functions $f$ for which the set $\sing(f^{-1})\setminus\{\infty\}$  is bounded.
The subclass of $\B$ for which $\sing(f^{-1})$ is finite is called the Speiser
class and denoted by~$\Speiser$.
We write $f\in \Speiser_q$ if the cardinality of $\sing(f^{-1})$ is equal to~$q$.
The dynamics of functions both in the Eremenko-Lyubich and Speiser class has
received much attention in recent years.
We refer to~\cite{Sixsmith2018} for an introduction to and 
a survey of recent results in these classes.

We recall the notion of order of growth. Let $T(r,f)$ be the Nevanlinna 
characteristic of a meromorphic function $f$;
see~\cite{Goldberg2008} for notations and results of Nevanlinna theory.
 The order $\rho(f)$ of $f$ is defined by
\begin{equation} 
\rho(f)=\limsup_{r\to\infty}\frac{\log T(r,f)}{\log r}.
\end{equation}
If $f$ is an entire function, one can replace $T(r,f)$ by $\log M(r,f)$ with $M(r,f)=\max_{|z|=r}|f(z)|$.

Many authors have studied the Hausdorff dimension of Julia sets and escaping sets.
We refer to~\cite{Kotus2008} and~\cite{Stallard2008} for surveys.
Let $\dim A$ denote the Hausdorff dimension of a set $A\subset\C$.
It was proved independently by
Bara\'nski~\cite{Baranski2008} and Schubert~\cite{Schubert2007}
that if $f\in\B$ is entire and of finite order, then $\dim I(f)=\dim J(f)=2$.
More generally, this result holds for meromorphic functions in $\B$ for which 
$\infty$ is an asymptotic value; see~\cite[Theorem~1.2]{Bergweiler2009}.

The situation is different for meromorphic functions for which $\infty$ is not an
asymptotic value, as shown by the following result~\cite[Theorem 1.1]{Bergweiler2012}. 

\begin{theorema}\label{thma}
Let $f\in \B$ be a transcendental meromorphic function of finite order~$\rho$.
Suppose that $\infty$ is not an asymptotic value of $f$ and that
there exists $M\in\N$ such that all but finitely many poles of $f$ have multiplicity 
at most~$M$. Then
\begin{equation} \label{1b}
\dim I(f)\leq \frac{2M\rho}{2+M\rho}.
\end{equation} 
\end{theorema}

We have strict inequality in~\eqref{1b} 
for functions of the form $R(e^z)$ \cite{Kotus2018} and
for Nevanlinna functions~\cite{Cui2019}. 
In general, however, \eqref{1b} is best possible.
More precisely, given $\rho\in[0,\infty)$ and $M\in\N$ there exists a transcendental 
meromorphic function $f\in\B$ satisfying the conditions stated in Theorem~\ref{thma}
such that we have equality in \eqref{1b}; see~\cite[Theorem 1.2]{Bergweiler2012}.

Our main result shows that~\eqref{1b} is also best possible in the Speiser class $\Speiser$ and even in $\Speiser_3$.

\begin{theorem}\label{theorem1}
Let $M\in\N$ and $\rho\in [0,\infty)$. 
Then there exists a function $f\in \Speiser_3$ of order $\rho$
such that all poles of $f$ have multiplicity $M$ and
\begin{equation} \label{1c}
\dim I(f)= \frac{2M\rho}{2+M\rho}.
\end{equation} 
\end{theorem}

The following result is an analogue of Theorem~\ref{theorem1} for the case $\rho=\infty$.
\begin{theorem}\label{theorem2}
Let $M\in\N$.
Then there exists a function $f\in \Speiser_3$ of infinite order
such that all poles of $f$ have multiplicity $M$ and
$\dim I(f)=2$.
\end{theorem}

Our results can be seen as a contribution to recent efforts to understand the 
differences and similarities between the classes $\B$ and~$\Speiser$.
We refer to~\cite{AspenbergCui,AspenbergCui2,BergweilerCui,Bishop2015,Bishop2015a,Bishop2017,Epstein2015} and the survey~\cite{Sixsmith2018} for results in this direction.
As an example we mention the recent result of Albrecht and Bishop~\cite{Albrecht2020} who showed
that given $\delta>0$ there exists an entire function $f\in\Speiser$ such that $1<\dim J(f)<1+\delta$.
Entire functions in $\B$ with this property had been constructed previously 
by Stallard~\cite{Stallard1997}.

As a consequence of our main result, we also have the following theorem which completes the study begun in~\cite{AspenbergCui} and continued in~\cite{AspenbergCui2}.

\begin{corollary}\label{cor1}
For any $d\in [0,2]$ there exists a function $f\in \Speiser_3$ such that $\dim I(f)=d$.
\end{corollary}

If $f\in\Speiser_2$, then $f$ has the form $f(z)=L(e^{az})$ where $a\in\C\setminus\{0\}$ and
$L$ is a fractional linear transformation. This implies that 
if $f\in\Speiser_2$, then $\dim I(f)$ is either $1/2$ or $2$; see~\cite{Kotus2018} and~\cite{McMullen1987}. So the least number of singular values to attain each possible Hausdorff dimension of escaping sets in the Speiser class is~$3$. 

A major difference to \cite{AspenbergCui, AspenbergCui2} is that we do not use quasiconformal surgery.
Our main tools are a recent result of Mayer and Urba\'nski~\cite{Mayer2021}
on the dimension of escaping sets, a method of MacLane~\cite{MacLane1947}
 and Vinberg~\cite{Vinberg1989} to construct real 
entire functions with preassigned critical values, and a result of 
Warschawski~\cite{Warschawski1967} on the angular derivative.

It was shown in~\cite{BergweilerCui} that for all $d\in (0,2]$ there exists a function $f\in \Speiser_3$ with $\dim J(f)=d$. Note that one has $I(f)\subset J(f)$ if $f\in\B$, as shown by Eremenko and Lyubich~\cite[Theorem~1]{Eremenko1992}
for entire $f$ and Rippon and Stallard~\cite[Theorem A]{Rippon1999} for meromorphic~$f$.
We also have $\dim J(f)>0$ for every transcendental meromorphic function $f$ by a result of Stallard~\cite{Stallard1994}.

Combining the methods of~\cite{BergweilerCui} with those of the present paper we obtain the following 
generalization of Corollary~\ref{cor1}.
\begin{theorem}\label{theorem3}
Let $0\leq d_I\leq d_J\leq 2$ and $d_J>0$.
Then there exists a function $f\in \Speiser_3$ such that
$\dim I(f)=d_I$ and $\dim J(f)=d_J$.
\end{theorem}

\section{Preliminary results} \label{prelim}
\subsection{Dimension of the escaping set} \label{subsec1}
Let $f\in\B$ and $M\in\N$ be as in Theorem~\ref{thma}.
Since $\infty$ is not an asymptotic value of~$f$, Iversen's theorem~\cite[Chapter~5, Theorem 1.1]{Goldberg2008} yields that 
$f$ has infinitely many poles.
Let $(a_j)$ be the sequence  of poles of $f$ and let $m_j$  be the   multiplicity of~$a_j$.
Choosing $M$ minimal we thus have
\begin{equation} \label{5a0}
M=\limsup_{j\to\infty} m_j< \infty.
\end{equation}

Let $b_j\in \C \setminus\{0\}$ be such that
\begin{equation} \label{5a}
f(z)\sim\left(\frac{b_j}{z-a_j}\right)^{m_j} \quad \text{as}\ z\to a_j .
\end{equation}
One of the ingredients in the proof of Theorem~\ref{thma} 
is the following result~\cite[Lemma~3.1]{Bergweiler2012}.
\begin{lemma}\label{lemma2}
Let $f$, $(a_j)$, $(b_j)$ and $M$ be as above.
If 
\begin{equation} \label{5b}
t>  \frac{2M \rho(f)}{2+ M \rho(f)},
\end{equation}
then
\begin{equation} \label{5c}
\sum_{j=1}^\infty\left(\frac{|b_j|}{|a_j|^{1+1/M}}\right)^t < \infty .
\end{equation}
\end{lemma}
The following result of Mayer and Urba\'nski~\cite{Mayer2021} underlines
the relevance of this series.
\begin{lemma}\label{mayer-urbanski}
Let $f$, $(a_j)$, $(b_j)$ and $M$ be as before.
Then $\dim I(f)$ is the infimum of the set of all $t>0$ for which~\eqref{5c} holds.
\end{lemma}

For the proof of Theorem~\ref{theorem3} we will require the following
addendum to Lemma~\ref{lemma2}.
\begin{lemma}\label{lemma2b}
Let $f$, $(a_j)$, $(b_j)$ and $M$ be as before.
Suppose that
\begin{equation} \label{5f}
T(r,f)=\O\!\left(\frac{r^{\rho}}{(\log r)^p}\right)  
\end{equation}
for some 
\begin{equation} \label{5f1}
p>\frac{4+M\rho}{2}.
\end{equation}
Then~\eqref{5c} also holds for
\begin{equation} \label{5e1}
t=  \frac{2M \rho}{2+ M \rho}.
\end{equation}
\end{lemma}
\begin{proof}
We proceed as in the proof of~\cite[Lemma~3.1]{Bergweiler2012},
but with $t$ given by~\eqref{5e1}.
In the terminology used there we then have $s=1$ and $\alpha=\rho$.
Let $n(r)=n(r,f)$ be the number of poles of $f$ in $\{z\colon |z|\leq r\}$.
As in~\cite{Bergweiler2012} we put
\begin{equation} \label{5e2}
P(l)=\left\{j\in\N:  n\left(2^l\right) \leq j < n\left(2^{l+1}\right)\right\}
= \left\{j\in\N:  2^l\leq |a_j|< 2^{l+1} \right\}
\end{equation}
and
\begin{equation} \label{5e3}
S_l = \sum_{j\in P(l)}\left(\frac{|b_j|}{|a_j|^{1+1/M}}\right)^t
= \sum_{j\in P(l)}\left(\frac{|b_j|}{|a_j|}\right)^t
\left(\frac{1}{|a_j|}\right)^{t/M}.
\end{equation}
H\"older's inequality then yields that
\begin{equation} \label{5g}
S_l\leq \left( \sum_{j\in P(l)}
\frac{|b_j|^2}{|a_j|^{2}}\right)^{t/2} \left( \sum_{j\in P(l)}
\frac{1}{|a_j|^{\rho}}\right)^{(2-t)/2}.
\end{equation}
The first term on the right hand side of~\eqref{5g} we estimate similarly as in~\cite{Bergweiler2012}. We find that
\begin{equation} \label{5h}
\begin{aligned}
\left(\sum_{j\in P(l)} \frac{|b_j|^2}{|a_j|^{2}}\right)^{t/2} 
&
\leq \frac{1}{(2^{2l})^{t/2}} \left( \sum_{j\in P(l)} |b_j|^2\right)^{t/2}
\\ &
\leq \frac{1}{2^{lt}} \left(36 R^2 2^{2(l+1)} \right)^{t/2}
=(12R)^t.
\end{aligned}
\end{equation}
To estimate the second term on the right hand side of~\eqref{5g},
we note that by standard estimates of Nevanlinna theory we have
\begin{equation} \label{5k}
n(r)\leq \int_r^{er} \frac{n(t)}{t} dt 
\leq N(er,f)\leq T(er,f)
=\O\!\left(\frac{r^{\rho}}{(\log r)^p}\right)  ,
\end{equation}
say $n(r)\leq C r^\rho/ (\log r)^p$ for large $r$ and some constant $C>0$.
Thus
\begin{equation} \label{5l}
\begin{aligned}
 \sum_{j\in P(l)} \frac{1}{|a_j|^{\rho}}
&= \int_{2^l}^{2^{l+1}} \frac{1}{t^{\rho}} dn(t)
= \frac{n(2^{l+1})}{(2^{l+1})^{\rho}}  
- \frac{n(2^{l})}{(2^{l})^{\rho}}  
+\rho\int_{2^l}^{2^{l+1}} \frac{n(t)}{t^{\rho+1}} dt
\\ &
\leq \frac{C}{(\log 2^{l+1})^p}
+C\rho \int_{2^l}^{\infty} \frac{1}{t (\log t)^p} dt
\\ &
= \frac{C}{(\log 2)^p (l+1)^p}
+\frac{C\rho}{(p-1)(\log 2)^{p-1} l^{p-1}}
=\O\!\left( \frac{1}{l^{p-1}}\right).
\end{aligned}
\end{equation}
Hence 
\begin{equation} \label{5m}
\left( \sum_{j\in P(l)} \frac{1}{|a_j|^{\rho}}\right)^{(2-t)/2}
=\O\!\left( \frac{1}{l^{(p-1)(2-t)/2}}\right).
\end{equation}
Combining this with~\eqref{5g} and~\eqref{5h} we thus find that
\begin{equation} \label{5n}
S_l=\O\!\left( \frac{1}{l^{(p-1)(2-t)/2}}\right).
\end{equation}
By~\eqref{5f1} and~\eqref{5e1} we have
\begin{equation} \label{5o}
\frac{(p-1)(2-t)}{2}>\frac12\left( \frac{4+M\rho}{2}-1\right)
\left(2-
\frac{2M \rho}{2+ M \rho}\right) =1 .
\end{equation}
Thus 
\begin{equation} \label{5p}
\sum_{l=0}^\infty S_l<\infty,
\end{equation}
from which the conclusion follows.
\end{proof}

\subsection{Growth of composite meromorphic functions} \label{subsec2}
For the following result we refer to~\cite[Satz 5.7]{Bergweiler1984}
and~\cite[Corollary~4]{Bergweiler1990}.
\begin{lemma}\label{lemma1}
Let $g$ be an entire and $f$ a meromorphic function. Then
\begin{equation} \label{b1}
\rho(g)\liminf_{r\to\infty}\frac{\log T(r,f)}{\log \log r}
\leq
\rho(f\circ g)
\leq
\rho(g)\limsup_{r\to\infty}\frac{\log T(r,f)}{\log \log r}. 
\end{equation}
\end{lemma}
In~\cite[Satz 5.7]{Bergweiler1984} the result is proved under the additional
hypothesis that
\begin{equation} \label{b2}
\limsup_{r\to\infty}\frac{\log T(r,f)}{\log \log r}<\infty.
\end{equation}
This hypothesis will be satisfied in our application.

In the proof of Theorem~\ref{theorem3} we will also use the following
result~\cite[Satz~2.2]{Bergweiler1984}.
\begin{lemma}\label{lemma1a}
Let $g$ be an entire and $f$ a meromorphic function. Then
\begin{equation} \label{b2a}
T(r,f\circ g)\leq (1+o(1)) T(M(r,g)+2|g(0)|,f).
\end{equation}
\end{lemma}
We note that the stronger inequality
\begin{equation} \label{b2b}
T(r,f\circ g)\leq (1+o(1)) \frac{T(r,g)}{\log M(r,g)}T(M(r,g),f).
\end{equation}
was proved in~\cite[Theorem~1]{Bergweiler1988}, but~\eqref{b2a} suffices for 
our purposes.

\subsection{A method of MacLane and Vinberg} \label{subsec3}
We will use a method of MacLane~\cite{MacLane1947} to construct a real entire
 function with a given sequence of real critical values.
(An entire function $g$ is called real if $g(\R)\subset\R$.)
For an exposition of the method we refer to Eremenko and Sodin~\cite{Eremenko1992a},
who in turn refer to Vinberg~\cite{Vinberg1989}. See also the paper by Eremenko and Yuditskii \cite{Eremenko2012}.
For a brief discussion of the method and an application of it in
complex dynamics we also mention~\cite[Section~5]{Bergweiler2015}.

Let $(c_n)_{n\in\Z}$ be a sequence
satisfying $(-1)^n c_n\geq 0$ for all $n\in\Z$. We consider the simply-connected 
``comb domain''
\begin{equation} \label{b3}
\Omega:=
\C\backslash \bigcup_{n\in\Z} \{x+in\pi\colon -\infty< x\leq \log|c_n|\}.
\end{equation}
Here we put $ \{x+in\pi\colon -\infty< x\leq \log|c_n|\}=\emptyset$ if $c_n=0$.
Assuming that not all $c_n$ are equal to~$0$, we have $\Omega\neq \C$.
Thus there exists a conformal map $\varphi$ mapping the lower half-plane
$\{z\in\C\colon \im z< 0\}$
onto $\Omega$.
We may choose $\varphi$ such that 
$\re \varphi(iy)\to +\infty$ as $y\to-\infty$.

The function $g:=\exp\circ \varphi$ extends continuously to $\R$ such that 
$g(x)\in\R$ for $x\in\R$. The Schwarz reflection principle now yields that $g$ 
extends to a real entire function.
The construction yields that all critical points of $g$ are real and that $(c_k)$ is 
the sequence of critical values of~$g$.

The function $g$ obtained this way belongs to the Laguerre-P\'olya class.
In particular, this implies that $\rho(g)\leq 2$.

\subsection{Conformal mappings of strips} \label{subsec4}
For our application of the MacLane-Vinberg method we will have to consider
the asymptotic behavior of~$\varphi$. Essentially, we will need that an
appropriate power of $\varphi$ has an ``angular derivative''.
In order to prove this we will use 
the following result of Warschawski~\cite[Theorem~5]{Warschawski1967}.

\begin{lemma}\label{warschawski}
Let $V$ be a simply-connected domain with the property
 that for all $\psi\in (0,\pi/2)$ there exists 
$R_\psi>0$ such that 
\begin{equation} \label{b8a}
V_\psi:=\{z\colon \re z>R_\psi,\; |\im z|<\psi\}\subset V.
\end{equation}
Let 
\begin{equation} \label{b8c}
S:=\left\{z\colon |\im z|<\frac{\pi}{2}\right\}
\end{equation}
and let $h\colon V\to S$ be a conformal map such that $\re h(x)\to\infty$
as $x\to\infty$.

Let $x_0\in\R$ be such that $[x_0,\infty)\subset V$ and,
for $x\geq x_0$, let $\theta(x)\in (0,\infty]$ be the length of the line segment of
$V\cap \{x+iy\colon y\in\R\}$ that contains~$x$. 
Put $\varepsilon(x):=\max\{\theta(x)-\pi,0\}$ and suppose that
\begin{equation} \label{b8}
\int_{x_0}^\infty \varepsilon(x)dx<\infty.
\end{equation}
Then there exists $\lambda\in (-\infty,+\infty]$ such that if $\psi\in (0,\pi/2)$,
then $h(z)-z\to \lambda$ as $\re z\to\infty$, uniformly for $z\in V_\psi$.
\end{lemma}
In our application the domain $V$ will be symmetric with respect to 
the real axis and it will contain the strip~$S$. 
Thus $\varepsilon(x)=\theta(x)-\pi$ in this case.
Moreover, we have $h(x)\in\R$ for $x\in\R$.
We can then deduce from the Schwarz lemma that $h'(x)<1$. 
It follows that in this special case we have $\lambda<\infty$.

Then Warschawski's result takes the following form.
\begin{lemma}\label{warschawski2}
Let $V$ be a simply connected domain which is symmetric with respect to the 
real axis and contains the strip~$S$.
Let $h\colon V\to S$ be conformal with $\re h(x)\to\infty$
as $x\to\infty$ and for $x\in\R$ let $\theta(x)\in (0,\infty]$ be the
length of the line segment of
$V\cap \{x+iy\colon y\in\R\}$ that contains~$x$. 
Suppose that 
\begin{equation} \label{b8b}
\int_{0}^\infty (\theta(x)-\pi )dx<\infty.
\end{equation}
Then there exists $\lambda\in\R$ such that if $\delta\in (0,\pi/2)$,
then $h(z)-z\to \lambda$ as $\re z\to\infty$, uniformly for $|\im z|\leq\pi/2-\delta$.
\end{lemma}

\section{Proof of Theorem~\ref{theorem1}} \label{prooftheorem}
Let $a:=\exp(2\pi i/M)$ if $M\geq 2$ and let $a\in\C\setminus\{0,1\}$ be arbitrary if $M=1$.
Let $G$ be a conformal map from the triangle with vertices~$0$,  $\pi/2$ and
$(\pi+i\pi)/2$ onto the disk (or a half-plane) which contains the points~$0$, $1$ and~$a$
in its boundary,
such that
$G(0)=0$, $G((\pi+i\pi)/2)=1$ and $G(\pi/2)=a$.
By reflection, $G$ extends to an even elliptic function with periods $\pi$ and $i\pi$.
The critical points of $G$ are at the points $m\pi/2+in\pi/2$ with $m,n\in\Z$.
Moreover, $G$ has the three critical values~$0$, $1$ and~$a$,
with all $a$-points being of multiplicity~$2$.

One can write $G$ in terms of the Weierstrass $\wp$-function
with periods $\pi$ and $\pi i$ as follows.
First we note that for a square lattice we have $e_2=0$ and $e_1=-e_3$,
using the standard notation $e_1=\wp(\pi/2)$, $e_2=\wp((\pi+i \pi)/2)$ and 
$e_3=\wp(i\pi/2)$ for the (finite) critical values of~$\wp$.
It follows that the critical values of $\wp^2$ are given by
$0$, $e_1^2$ and~$\infty$. Choosing  $L$ as the fractional linear transformation
satisfying $L(\infty)=0$, $L(0)=1$ and $L(e_1^2)=a$  we then find that 
$G=L(\wp^2)$.

Let $H:=G^M$. Then $H$ is an even elliptic function. If $M\geq 2$, then $H$
has the critical values~$0$, $1$ and $\infty$ and all poles of $H$ have multiplicity~$M$.
If $M=1$ so that $H=G$, then $H$ has critical values~$0$, $1$ and $a$ and the poles 
are simple. Thus in any case we have $H\in\Speiser_3$ and the poles of $H$ have  multiplicity~$M$.

We now consider 
\begin{equation} \label{3c}
F(z):=H(\arcsin z) 
\end{equation} 
and note that this defines a meromorphic function $F\in \Speiser_3$.
Examples of this type were previously considered 
in~\cite[p.~734]{Teichmueller1944}, \cite[Section~5]{Bank1976},
\cite[Section~2]{Langley2002}, \cite{Eremenko2004}
and~\cite{BergweilerCui}.

We note that the points $z_k:=\pi/2+ik\pi/2$ are critical points of~$H$.
It follows that the points 
$x_k :=\sin( \pi/2+ik\pi/2)=\cosh(k\pi/2)$
are critical points of~$F$.
Note that if $y>0$, then 
\begin{equation} \label{3e}
\cosh(x+y)=\frac12\left(e^{x+y}+e^{-x-y}\right)\leq\frac12\left(e^{x+y}+e^{-x+y}\right) 
=e^y\cosh(x)
\end{equation} 
for all $x\in\R$. Thus 
\begin{equation} \label{3f}
\frac{x_{k+1}}{x_k} \leq e^{\pi/2} 
\end{equation} 
for all $k\in\N$.
Since $F$ is even, we also have the critical points $x_0:=0$ and $x_{-k}:=-x_k$ for $k\in\N$.

Since the elliptic function $H$ has order~$2$, and $H(z)=F(\sin z)$,
a result of Edrei and Fuchs~\cite[Corollary~1.2]{Edrei1964} yields that $F$ has order~$0$.
Together with Theorem~\ref{thma}
this implies that if $\rho=0$, then $f:=F$ satisfies the conclusion of our theorem.

In the case that $\rho>0$ we will define $f:=F\circ g$ for a suitable entire function~$g$. 
Before starting with the construction of $g$ we note that much more precise information 
about the growth of $F$ is known, not just that it has order~$0$. 
In fact, one can show as in the papers cited above that 
\begin{equation} \label{3g}
T(r,F)\sim c( \log r)^2
\end{equation} 
for some positive constant $c$ as $r\to\infty$.
Together with Lemma~\ref{lemma1} this yields that  
\begin{equation} \label{3h}
\rho(f)=\rho(F\circ g)=2\rho(g). 
\end{equation} 
We shall thus construct a suitable entire function $g$ of order $\alpha:=\rho/2$.

First we restrict to the case that $0<\alpha<1$. 
For $k\in\Z$ we define $j(k)$ as the maximal integer such that 
$\{i\pi k+x\colon -\infty<x\leq \log x_{j(k)}\}$ does not intersect the sector
\begin{equation} \label{3k}
W_\alpha:=\left\{z\colon|\arg z|<\frac{\alpha\pi}{2}\right\}.
\end{equation} 
Equivalently, 
\begin{equation} \label{3i}
j(k) =\max\left\{n\in\Z\colon \log x_n \leq \pi |k|\cot\frac{\alpha\pi}{2}\right\}.
\end{equation} 
We will apply the MacLane-Vinberg method with $c_k:=(-1)^k x_{j(k)}$ and thus
consider the domain $\Omega$ defined by~\eqref{b3}.
Since $\Omega$ is symmetric with respect to the real axis, we may choose the
conformal map $\varphi$ from the lower half-plane to $\Omega$
such that $\varphi(iy)\in\R$ for $y<0$.
By construction, $\Omega$ contains the sector $W_\alpha$.
On the other hand, it follows from~\eqref{3f} that the endpoints of the lines 
that form the boundary of $\Omega$ have distance at most $\pi/2$ from~$\partial W_\alpha$.

For $r>0$, let $\psi(r)$ be the angular measure of the arc of the
circle $\{z\colon |z|=r\}$ which is contained in~$\Omega$.
It follows from the above considerations that there exists $K>0$ such that
\begin{equation} \label{6a}
\alpha\pi\leq \psi(r) \leq\alpha\pi+\frac{K}{r}.
\end{equation}
Let $V$ be the component of 
the preimage of $\Omega$ under $z\mapsto e^{\alpha z}$ which contains $(0,\infty)$.
Thus $V$ is the image of $\Omega$ under the map $z\mapsto (\log z)/\alpha$, 
with the principal branch of the logarithm.
Let $\theta$ be defined as in Lemma~\ref{warschawski2}. Then 
\begin{equation} \label{6b0}
\theta(x)= \frac{1}{\alpha}\psi(e^{\alpha x})
\end{equation}
and hence
\begin{equation} \label{6c}
\pi \leq \theta(x) \leq\pi +\frac{K}{\alpha e^{\alpha x}}.
\end{equation}
We will apply Lemma~\ref{warschawski2} to the conformal map
\begin{equation} \label{6c1}
h\colon V\to S,\quad 
h(w)=\log\!\left( i\varphi^{-1}(e^{\alpha w})\right).
\end{equation}
Note that by~\eqref{6c} the hypotheses of this lemma
are satisfied.
Hence there exists $\lambda\in\R$ 
such that for any $\delta>0$ we have
\begin{equation} \label{6d}
h(w)=w+\lambda+o(1)
\quad\text{as}\ \re w\to\infty, \  |\im w|\leq \frac{\pi}{2}-\delta.
\end{equation}
We may normalize $\varphi$ and hence $h$ so that $\lambda=0$.
Thus 
\begin{equation} \label{6d1}
 \varphi^{-1}(e^{\alpha w}) \sim -i e^w
\quad\text{as}\ \re w\to\infty, \  |\im w|\leq \frac{\pi}{2}-\delta.
\end{equation}
Putting $z=e^w$ this takes the form 
\begin{equation} \label{6d2}
 \varphi^{-1}(z^\alpha) \sim -i z
\quad\text{as}\ |z|\to\infty, \  |\arg z|\leq \frac{\pi}{2}-\delta.
\end{equation}
This implies that 
\begin{equation} \label{6e}
\varphi(-iz)\sim z^\alpha
\quad \text{as}\ |z|\to\infty,\ |\arg z|\leq \frac{\pi}{2}-\delta .
\end{equation}
Recalling that $g$ is defined by $g(z)=\exp\varphi(z)$ for $\im z<0$ we deduce that 
\begin{equation} \label{6e1}
\log M(r,g)\geq \log|g(-ir)|=\re \varphi(-ir)\geq (1-o(1))r^\alpha
\end{equation}
so that $\rho(g)\geq\alpha$.

To prove that $\rho(g)\leq\alpha$ we 
note that, by~\eqref{6e}, 
\begin{equation} \label{6e2}
\log |g(z)|\leq (1+o(1))|z|^\alpha
\quad\text{as}\ |z|\to\infty, \ -\pi+\delta \leq \arg z \leq -\delta.
\end{equation}
Since $g$ is symmetric with respect to the real axis,
the same estimate holds for $\delta \leq \arg z \leq \pi-\delta$.

It remains to estimate $g$ in the sectors $T:=\{z\colon |\arg z|\leq \delta\}$ and $-T$.
To this end we consider the function
$k\colon T\to \C$, $k(z)=\exp(-z^\alpha)g(z)$.
It follows from~\eqref{6e} that $k(re^{-i\delta})\to 0$ as $r\to\infty$.
By symmetry we also have $k(re^{i\delta})\to 0$ as $r\to\infty$.
Thus $k$ is bounded on $\partial T$. 
On the other hand, since $g$ has order at most~$2$, we 
have $\log |k(z)|\leq \log |g(z)|\leq |z|^{3}$ if $z\in T$ and $|z|$ is 
sufficiently large.
Choosing $\delta<\pi/6$ we deduce from the Phragm\'en-Lindel\"of 
principle~\cite[Chapter~I, Theorem~21]{Levin1964} that 
$k$ is bounded in~$T$. This implies 
that $\log |g(z)|\leq (1+o(1))|z|^\alpha$ as $|z|\to\infty$, $z\in T$.
The same argument can be made for the sector $-T$. 
Altogether we obtain $\log|g(z)|\leq (1+o(1))|z|^\alpha$ as $|z|\to\infty$, with no
restriction on $\arg z$, which together with~\eqref{6e1} yields that 
\begin{equation} \label{6e4}
\log M(r,g)\sim r^\alpha
\quad \text{as}\ r\to\infty.
\end{equation}
In particular, 
\begin{equation} \label{6e3}
\rho(g)= \alpha=\frac{\rho}{2}.
\end{equation}
In view of~\eqref{3h} the function $f=F\circ g$ thus satisfies $\rho(f)=\rho$.

By construction, the critical values of $g$ are critical points of~$F$. This implies
that $f$ and $F$ have the same critical values.
Thus $f$ has three critical values. Next we note that
$g$ has no asymptotic value. For example,
this follows since $g$ is even and can thus be written in the form $g(z)=g_0(z^2)$ 
for some entire function~$g_0$. We have $\rho(g_0)=\rho(g)/2=\alpha/2<1/2$.
A classical result of Wiman now implies that $g_0$ is unbounded on any curve 
tending to~$\infty$. In particular, $g_0$ has no asymptotic value.
It follows that $g$ and hence $f$ have no asymptotic value.
Altogether we thus find that $f\in\Speiser_3$.

Let $(a_j)$ be the sequence of poles of $f=F\circ g$. Since the poles have 
multiplicity $M$ there exists a sequence $(b_j)$ in $\C\setminus\{0\}$ such that
\begin{equation} \label{6f}
f(z) \sim \left(\frac{b_j}{z-a_j}\right)^M
\quad \text{as}\ z\to a_j .
\end{equation}
In order to apply Lemma~\ref{mayer-urbanski}
we want to estimate $|b_j|$ in terms of~$|a_j|$.

We begin by considering the corresponding terms for~$F$.
First we note that since $G$ is elliptic, the residues at the poles of $G$ can 
take only finitely many values. Thus their moduli are bounded below, say
$|\res(\zeta,G)|\geq C_0>0$ if $\zeta$ is a pole of~$G$.
It follows that if $(\alpha_j)$ is the sequence of poles of~$H$, 
and $\beta_j$ is such that 
\begin{equation} \label{6g}
H(z) \sim \left(\frac{\beta_j}{z-\alpha_j}\right)^M
\quad \text{as}\ z\to \alpha_j ,
\end{equation}
then $|\beta_j|\geq C_0$ for all $j\in\N$.

Let now $(A_j)$ be the sequence of poles of $F$ and let $B_j$ be such that
\begin{equation} \label{6h}
F(z) \sim \left(\frac{B_j}{z-A_j}\right)^M
\quad \text{as}\ z\to A_j .
\end{equation}
For each $j\in\N$ there exists  $\alpha_j\in\C$  such that $\sin\alpha_j=A_j$.
It follows that 
\begin{equation} \label{6i}
\begin{aligned}
H(z)
=F(\sin z) 
&\sim \left(\frac{B_j}{\sin z-A_j}\right)^M
= \left(\frac{B_j}{\sin z-\sin \alpha_j}\right)^M \\
& = \left(\frac{z-\alpha_j}{\sin z-\sin \alpha_j}\cdot\frac{B_j}{z-\alpha_j}\right)^M
\\ &
\sim\left(\frac{1}{\cos\alpha_j}\cdot\frac{B_j}{z-\alpha_j}\right)^M
\quad \text{as}\ z\to \alpha_j .
\end{aligned}
\end{equation}
Combining~\eqref{6g} and~\eqref{6i} we deduce that 
\begin{equation} \label{6j}
B_j^M
=\beta_j^M\cos^M\alpha_j .
\end{equation}
Since $\sin\alpha_j=A_j\to\infty$ as $j\to\infty$, we have 
\begin{equation} \label{6j1}
|\cos\alpha_j| =\sqrt{\left|1-\sin^2\alpha_j\right|}\sim |\sin\alpha_j|=|A_j|.
\end{equation}
Thus
\begin{equation} \label{6k}
|B_j|=|\beta_j\cos\alpha_j|\geq C_0|\cos\alpha_j|\geq 
(C_0-o(1)|A_j|
\end{equation}
as $j\to\infty$.
In particular, there exists a positive constant $C$ such that 
\begin{equation} \label{6k1}
|B_j|\geq C|A_j|
\end{equation}
for all~$j\in\N$.

Recall that $(a_j)$ is the sequence of poles of $f=F\circ g$ while 
$(A_j)$ is the sequence of poles of~$F$. It follows that for each $j\in\N$
there exists $k\in\N$ such that $g(a_j)=A_k$. 

Similarly as in~\eqref{6i} we see that,
as $z\to a_j$ and hence $g(z)\to A_k$,
\begin{equation} \label{6l}
\begin{aligned}
f(z)
=F(g(z)) &
\sim \left(\frac{B_k}{g(z)-A_k}\right)^M
=\left(\frac{B_k}{g(z)-g(a_j)}\right)^M
\\ &
\sim \left(\frac{B_k}{g'(a_j)(z-a_j)}\right)^M .
\end{aligned}
\end{equation}
Together with~\eqref{6f} and~\eqref{6k1} this yields that 
\begin{equation} \label{6m}
|b_j|
=\left|\frac{B_k}{g'(a_j)}\right|
\geq C\left|\frac{A_k}{g'(a_j)}\right|
=C\left|\frac{g(a_j)}{g'(a_j)}\right|
=\frac{C}{|\varphi'(a_j)|}. 
\end{equation}

It follows from~\eqref{6e} that if $0<\delta<\delta'<\pi/2$, then
\begin{equation} \label{6n}
-i\varphi'(-iz)\sim \alpha z^{\alpha-1}
\quad \text{as}\ |z|\to\infty,\ |\arg z|<\frac{\pi}{2}-\delta'.
\end{equation}
Combining~\eqref{6m} and~\eqref{6n} we see that, as $j\to\infty$,
\begin{equation} \label{6p}
|b_j|\geq (1-o(1)) \frac{C}{\alpha|a_j|^{\alpha-1}}
\quad\text{for}\ 
a_j\in \Delta:=\left\{re^{it}\colon -\frac{3\pi}{4}<t<-\frac{\pi}{4}\right\}.
\end{equation}

We are interested in exponents $t$ for which the series in Lemma~\ref{lemma2} diverges.
It follows from~\eqref{6p} that the series 
\begin{equation} \label{6p1}
\sum_{a_j\in \Delta}\left(\frac{|b_j|}{|a_j|^{1+1/M}}\right)^t 
\end{equation}
and hence the series in~\eqref{5b} diverge if 
\begin{equation} \label{6p2}
\sum_{a_j\in \Delta}\left(\frac{1}{|a_j|^{\alpha-1}|a_j|^{1+1/M}}\right)^t 
=\sum_{a_j\in \Delta}\frac{1}{|a_j|^{t(\alpha+1/M)}}=\infty.
\end{equation}

We thus have to estimate the exponent of convergence of the poles of $H$ in~$\Delta$.
In order to do so let $p$ be a pole of $G$ and hence of~$H$.
Then $p_n:=p-i\pi n$ is also a pole of~$H$, for all $n\in\N$.
Hence $q_n:=\sin p_n$ is a pole of~$F$.
Note that with $w:=p-\pi/2+i\log 2$ we have  
\begin{equation} \label{6q}
q_n=\frac{1}{2i}\left(e^{n\pi+ip}-e^{-n\pi-ip}\right)
=\frac{e^{n\pi+ip}}{2i}\left(1-e^{-2n\pi-i2p}\right)
=\exp(n\pi+iw+\delta_n)
\end{equation}
for some (complex) sequence $(\delta_n)$ tending to~$0$.
For $m\in\Z$ we now put
\begin{equation} \label{6r}
u_{m,n}=n\pi+i2m\pi+ iw+\delta_n .
\end{equation}
It follows that if $\varphi(z)=u_{m,n}$, then 
$g(z)=\exp\varphi(z)=\exp u_{m,n}=q_n$ so that $z$ is a pole of~$H$.
In other words, the preimages of the $u_{m,n}$ under $\varphi$ are poles of~$H$.

For $r>0$ we put 
\begin{equation} \label{6s}
\Delta_r:=\Delta\cap \left\{z\colon \frac12 r \leq |z|\leq r\right\}
=\left\{z \colon \frac12 r \leq |z|\leq r,\; -\frac{3\pi}{4}<\arg z<-\frac{\pi}{4}\right\}.
\end{equation}
It follows from~\eqref{6e} that
for large $r$ the map $\varphi$ is univalent in $\Delta_r$ and that given $\varepsilon>0$
we have
\begin{equation} \label{6t}
\varphi(\Delta_r)\supset 
\left\{z\colon \frac{1+\varepsilon}{2^\alpha} r^\alpha \leq |z|\leq (1-\varepsilon) r^\alpha,
\; |\arg z|<(1-\varepsilon) \frac{\alpha\pi}{4}\right\}.
\end{equation}
This easily yields that there exists a positive constant $\eta$ such that
if $r$ is sufficiently large, then
$\varphi(\Delta_r)$ contains at least $\eta r^{2\alpha}$ of the points $u_{m,n}$.
Thus the number of poles of $H$ in $\Delta_r$ is at least $\eta r^{2\alpha}$.
This implies (see, e.g., \cite[Chapter~2, Theorem 1.8]{Goldberg2008}) that 
\begin{equation} \label{6u}
\sum_{a_j\in \Delta}\frac{1}{|a_j|^{2\alpha}}=\infty.
\end{equation}
Recalling that $\alpha=\rho/2$ we conclude that~\eqref{6p2} holds for 
\begin{equation} \label{6v}
t=\frac{2\alpha}{\alpha+1/M} 
=\frac{2M\alpha}{M\alpha+1} 
=\frac{2M\rho}{2+M\rho} . 
\end{equation}
Lemma~\ref{mayer-urbanski} yields that $\dim I(f)\geq t$. 
As the opposite inequality follows from Theorem~\ref{thma}, we conclude that~\eqref{1c}
holds.

This proves the theorem in the case that $0<\rho<2$ so that $0<\alpha<1$.
Suppose now that $\rho\geq 2$. We put $N:=\lfloor\rho\rfloor$  and $\rho_0:=\rho/N$.
Then $1\leq \rho_0<2$.
Let $f_0\in\Speiser_3$ be the meromorphic function of order $\rho_0$ obtained by 
the above construction.
Thus $f_0=F\circ g_0$ for the entire function $g_0$ obtained by the MacLane-Vinberg method
with $\alpha=\rho_0/2$.
The construction yields that $f_0(0)=0$ and that $0$ is a critical value of~$f_0$.

Let $f_1:=f_0^N$. Then $\rho(f_1)=\rho(f_0)=\rho_0$. The function $f_1$ is obtained
the same way as the function $f_0$, except that the multiplicities of the poles of
$f_1$ is $NM$. Thus~\eqref{1c} takes the form 
\begin{equation} \label{6w}
\dim I(f_1)=\frac{2NM\rho(f_1)}{2+NM\rho(f_1)} 
=\frac{2NM\rho_0}{2+NM\rho_0} 
=\frac{2M\rho}{2+M\rho} . 
\end{equation}
We now put
\begin{equation} \label{6x}
f(z):=f_0(z^N) .
\end{equation}
Since $f_0\in\Speiser_3$, $f_0(0)=0$ and $0$ is a critical value of~$f_0$, we find that $f\in\Speiser_3$.
Moreover, we have $\rho(f)=N\rho(f_0)=\rho$.

With $P(z):=z^N$ we have 
\begin{equation} \label{6y}
P\circ f =f_1\circ P.
\end{equation}
It is not difficult to see that this implies that 
$I(f_1)= P(I(f))$ and hence
\begin{equation} \label{6z}
\dim I(f)=\dim I(f_1) .
\end{equation}
Combining this with~\eqref{6w} we see that~\eqref{1c} holds in this case as well.\qed

\section{Proof of Theorem~\ref{theorem2}} \label{prooftheorem2}
Let $H=G^M$ be as in Section~\ref{prooftheorem}. Let $p$ be a pole of $H$
and let $p_{m,n}=p+m\pi+in\pi$.
Then $p_{m,n}$ is a pole of $H$ for all $m,n\in\Z$ and there exists $\beta$ such that 
\begin{equation} \label{7a}
H(z) \sim \left(\frac{\beta}{z-p_{m,n}}\right)^M
\quad \text{as}\ z\to p_{m,n}.
\end{equation}
Let $a_{m,n,k}=\log p_{m,n} +2\pi i k$, with the principle branch of the logarithm.
Then $a_{m,n,k}$ is a pole of order $M$ of $f(z):=H(e^z)$.
Thus there exist $b_{m,n,k}$ such that 
\begin{equation} \label{7b}
f(z)\sim  \left(\frac{b_{m,n,k}}{z-{a_{m,n,k}}}\right)^M
\quad \text{as}\ z\to a_{m,n,k}.
\end{equation}
Since
\begin{equation} \label{7c}
\begin{aligned} 
f(z) 
&\sim \left(\frac{\beta}{e^z-p_{m,n}}\right)^M
= \left(\frac{\beta}{e^z-e^{a_{m,n,k}}}\right)^M
\\ &
\sim  \left(\frac{\beta}{e^{a_{m,n,k}}(z-{a_{m,n,k}})}\right)^M
=  \left(\frac{\beta}{p_{m,n}(z-{a_{m,n,k}})}\right)^M
\quad \text{as}\ z\to a_{m,n,k} ,
\end{aligned} 
\end{equation}
we have 
\begin{equation} \label{7d}
b_{m,n,k}=\frac{\beta}{p_{m,n}}.
\end{equation}
The series \eqref{5c} contains
\begin{equation} \label{7e}
\begin{aligned} 
\sum_{m,n,k=-\infty}^\infty\left(\frac{b_{m,n,k}}{|a_{m,n,k}|^{1+1/M}}\right)^t 
&=
\sum_{m,n,k=-\infty}^\infty\left(\frac{|\beta|}{|p_{m,n}|\cdot|a_{m,n,k}|^{1+1/M}}\right)^t 
\\ &
=|\beta|^t\sum_{m,n=-\infty}^\infty\frac{1}{|p_{m,n}|^t}
\sum_{k=-\infty}^\infty\frac{1}{|a_{m,n,k}|^{(1+1/M)t}} 
\end{aligned} 
\end{equation}
as a subseries.

By a suitable choice of the conformal map defining $G$ we may achieve that $G$ 
and hence $H$ have
no poles of modulus~$1$. This implies that there exists $\delta>0$ such that 
$|\log |p_{m,n}||\geq \delta$ for all $m,n\in\Z$.
Fix $m$ and $n$ and write $\log p_{m,n}=u+iv$ with $u,v\in\R$.
Then $|u|\geq\delta$ and $|v|\leq\pi$.
Hence, assuming that $\delta\leq 1/2$, we have
\begin{equation} \label{7e1}
\begin{aligned} 
|a_{m,n,k}|^2
&=u^2+(v+2\pi k)^2
\leq u^2+2(v^2+(2\pi k)^2)
\leq  u^2+2\pi^2 +8\pi^2 k^2
\\ &
\leq \left(1+\frac{2\pi^2}{\delta^2}\right)u^2+8\pi^2 k^2
\leq \left(1+\frac{2\pi^2}{\delta^2}\right)(u^2+ k^2).
\end{aligned} 
\end{equation}
With 
\begin{equation} \label{7e2}
A_t:=\left(1+\frac{2\pi^2}{\delta^2}\right)^{-(1+1/M)t/2}
\quad\text{and}\quad
B_t:=A_t  \int_{0}^\infty\frac{dy}{|1+y^2|^{(1+1/M)t/2}}
\end{equation}
we thus have
\begin{equation} \label{7e3}
\begin{aligned} 
\sum_{k=-\infty}^\infty\frac{1}{|a_{m,n,k}|^{(1+1/M)t}} 
&\geq 
A_t \sum_{k=-\infty}^\infty\frac{1}{|u^2+k^2|^{(1+1/M)t/2}} 
\\ &
\geq 
A_t \int_{0}^\infty\frac{dx}{|u^2+x^2|^{(1+1/M)t/2}} 
=\frac{B_t}{|u|^{(1+1/M)t-1}}.
\end{aligned} 
\end{equation}
Recalling that $u=\log|p_{m,n}|$ 
and noting that given $\varepsilon>0$ there exists $C>0$ such that 
$|\log|p_{m,n}||^{(1+1/M)t-1}\leq C|p_{m,n}|^\varepsilon$ 
for all $m,n\in\Z$ 
we deduce that 
\begin{equation} \label{7e4}
\sum_{k=-\infty}^\infty\frac{1}{|a_{m,n,k}|^{(1+1/M)t}} 
\geq \frac{B_t}{|\log|p_{m,n}||^{(1+1/M)t-1}}
\geq \frac{B_t}{C |p_{m,n}|^{\varepsilon}}.
\end{equation}
Together with~\eqref{7e} we thus have 
\begin{equation} \label{7e5}
\sum_{m,n,k=-\infty}^\infty\left(\frac{b_{m,n,k}}{|a_{m,n,k}|^{1+1/M}}\right)^t 
\geq 
\frac{|\beta|^t B_t}{C}\sum_{m,n=-\infty}^\infty\frac{1}{|p_{m,n}|^{t+\varepsilon}}.
\end{equation}
Since 
\begin{equation} \label{7j}
\sum_{m,n=-\infty}^\infty\frac{1}{|p_{m,n}|^2}
\end{equation}
diverges, the series on the left hand side of~\eqref{7e} and~\eqref{7e5}
  diverges for all $t\in (0,2)$.
Hence the series in~\eqref{5c} diverges for all $t\in (0,2)$.
Lemma \ref{mayer-urbanski} now shows that $\dim I(f)=2$.\qed

\section{Proof of Theorem~\ref{theorem3}} \label{prooftheorem3}
Suppose first that $d_J=2$. If $d_I=2$, then the function constructed in Theorem~\ref{theorem2} has
the desired property. Thus let $0\leq d_I<2=d_J$. We put $M=2$ and choose $\rho\in [0,\infty)$ 
such that 
\begin{equation} \label{4a}
\frac{2M\rho}{2+M\rho}=d_I.
\end{equation}
Let $f$ be the function constructed in Theorem~\ref{theorem1}.
Then $\dim I(f)=d_I$. However, for $\alpha\in\C\setminus\{0\}$ and $\beta\in\C$ we also have
$\dim I(\alpha f+\beta )=d_I$. To see this we recall that the essential point in the proof was
to determine for which parameters $t$ the series~\eqref{6p1} converges. This does not change
when we replace $f$ by $\alpha f+\beta$.

We shall show that for a suitable choice of the parameters $\alpha$ and $\beta$ we have
$J(\alpha f+\beta )=\C$ and thus $\dim J(\alpha f+\beta )=2=d_J$.
In order to do so we recall that $f$ has infinitely many poles $a_1,a_2,\dots$.
We put $\alpha=a_2-a_1$ and $\beta=a_1$ so that $f_0:=\alpha f+\beta = (a_2-a_1)f+a_1$.
Since $f$ has the critical values $0$, $1$ and $\infty$, we find that $f_0$ has 
the critical values $a_1$, $a_2$ and~$\infty$.
Thus all critical values of $f_0$ are poles. Since $f$ and hence $f_0$
have no asymptotic values, 
standard results of complex dynamics relating periodic components of the Fatou set
to singularities of the inverse~\cite[\S 4.3]{Bergweiler1993}
now show that $f_0$ has no attracting or parabolic basins, no Siegel disks, and no Herman rings.
Since functions in the Speiser class have no wandering domains and no Baker
domains~\cite{Baker1992,Rippon1999}, it follows 
that $J(f)=\C$.
This completes the proof in the case that $d_J=2$.

Suppose now that $0<d_J<2$. 
First we consider the case that $d_I<d_J$ so that $0\leq d_I<d_J<2$. We choose $M\in\N$ such 
that 
\begin{equation} \label{4b}
\frac{2M}{1+M}>d_J
\end{equation}
and then choose $\rho\in[0,2)$ such that~\eqref{4a} holds.

We proceed as in the proof of Theorem~1.1, but with two modifications:
\begin{itemize}
\item[$(a)$]
For a large integer $N$ to be determined, we define $c_k=(-1)^k$ for $|k|\leq N$, 
and define $c_k=(-1)^k x_{j(k)}$, with $j(k)$ given by~\eqref{3i}, only for $|k|>N$.
We denote the resulting entire function by $g_N$.
\item[$(b)$]
We replace the elliptic function $H=G^M$ by $H_\kappa(z)=H(\kappa z)$, where 
$\kappa$ is chosen so small that $H_\kappa(D(0,4))\subset D(0,1)$.
\end{itemize}
If we apply the MacLane-Vinberg method to the sequence $((-1)^k)_{k\in\Z}$, then 
the resulting function $g$ has the form $g(z)=\cos(\alpha z+\beta)$ where $\alpha ,\beta\in\C$, with $\alpha \neq 0$.
We had already noted that for symmetric sequences the resulting functions can assumed to be even.
It thus follows from $(a)$ that, as $N\to\infty$, we have $g_N(z)\to \cos(\alpha z)$ for some 
$\alpha\in\C\setminus\{0\}$.
We can normalize the maps $\varphi$ such that $\alpha=1$.

Since $\arcsin x=\pi/2-\arccos x$ this yields with $F_\kappa(z)=H_\kappa(\arcsin z)$ as in~\eqref{3c} that 
\begin{equation} \label{4d}
F_\kappa(g_N(z))= H_\kappa\!\left(\frac{\pi}{2}-\arccos g_N(z)\right) \to K(z):=H_\kappa\!\left(\frac{\pi}{2}-z\right)
\end{equation}
as $N\to\infty$.

It follows from $(b)$ that $K(D(0,2))\subset D(0,1)$. By~\eqref{4d} we also have
\begin{equation} \label{4e}
(F_\kappa\circ g_N)(D(0,2))\subset D(0,1)
\end{equation}
if $N$ is sufficiently large.
This implies that $K$ and $F_\kappa\circ g_N$ have attracting fixed points, whose
attracting basins
contain $D(0,2)$ and thus in particular the critical values $0$ and~$1$.
This implies~\cite[Lemma~2.7]{BergweilerCui} that both $F(K)$ and $F(F_\kappa\circ g_N)$ consist of 
a single attracting basin and that $J(K)$ and $J(F_\kappa\circ g_N)$ are totally disconnected.

A result of Kotus and Urba\'nski~\cite{Kotus2003} says that $\dim J(K)>2M/(1+M)$.
It follows from~\cite[Lemma~2.5]{BergweilerCui} that if $N$ is sufficiently large,
then we also have
\begin{equation} \label{4f}
\dim J(F_\kappa\circ g_N)>\frac{2M}{1+M}.
\end{equation}
For $N$ satisfying~\eqref{4e} and~\eqref{4f}
and $\lambda\in (0,1]$ we consider the function $f_\lambda$  defined by
\begin{equation} \label{4g}
f_\lambda(z)=F_\kappa(g_N(\lambda z)).
\end{equation}
The arguments in the proof of Theorem~\ref{theorem1} depend only on the 
behavior of $g(z)$ as $|z|\to\infty$. So they remain valid if $g$ is replaced by~$g_N$.
They also hold if $H$ and $f$ are replaced by $H_\kappa$ and $f_\lambda$.
Thus~\eqref{1c} holds for $f=f_\lambda$ and hence 
\begin{equation} \label{4g1}
\dim I(f_\lambda)=d_I
\end{equation}
for $\lambda\in (0,1]$ by~\eqref{4a}.

Using the same arguments as in~\cite{BergweilerCui} we find that the function
\begin{equation} \label{4h}
\lambda\mapsto \dim J(f_\lambda)
\end{equation}
is continuous and that 
\begin{equation} \label{4i}
\lim_{\lambda\to 0} \dim J(f_\lambda) =d_I<d_J.
\end{equation}
Since $f_1=F_\kappa\circ g_N$ it thus follows from~\eqref{4b} and~\eqref{4f} that
$\dim J(f_1)>d_J$.
Thus there exists $\lambda\in(0,1)$ such that $\dim J(f_\lambda) =d_J$.
This completes the proof in the case that $d_I<d_J$.

It remains to consider the case that $d_I=d_J\in (0,2)$.
Here we choose $M\in\N$ and $\rho\in (0,2)$ such that 
\begin{equation} \label{4j}
\frac{2M\rho}{2+M\rho}=d_I=d_J.
\end{equation}
Again we put $\alpha:=\rho/2$.
The maximum modulus of the function $g$ considered in the proof of 
Theorem~\ref{theorem1} has the asymptotics~\eqref{6e4}.
We will modify the construction to achieve that instead we have
\begin{equation} \label{4l}
\log M(r,g)\sim \frac{r^\alpha}{(\log r)^{2q}}
\quad \text{as}\ r\to\infty
\end{equation}
for some (large) $q\in\N$.
In order to do so we recall that $z\mapsto E_\alpha(z):=e^{\alpha z}$ maps the strip 
$T:=\{z\colon |\im z|<\pi/\alpha\}$ conformally onto $\C\setminus (-\infty,0]$
and that $V$ was defined as the component of the preimage of 
the comb domain $\Omega$ under $E_\alpha$ which contains the real axis.
Here the definition of $\Omega$ was made in such a way that 
that it contains the sector $W_\alpha=\{z\colon |\arg z|<\alpha\pi/2\}$, which is the 
image of the strip $S=\{z\colon |\im z|<\pi/2\}$ under~$E_\alpha$.
Consequently, $V\supset S$.

We now replace the map $E_\alpha$ by 
\begin{equation} \label{4m}
z\mapsto E_\alpha^*(z):=\frac{e^{\alpha z}}{(z^2+c^2)^q},
\end{equation}
where $c$ is a large positive constant and $q\in\N$ is also large.
If $c$ is chosen sufficiently large, then there exists a domain
$T^*$ containing the strip $S$ such that $E_\alpha^*$ maps $T^*$ conformally
onto $\C\setminus (-\infty,0]$. 
To see this, we note that given $\delta>0$ and a horizontal strip of
height $\pi-\delta$ we can choose $c$ such that $(E_\alpha^*)'(z)$ is contained 
in some half-plane for all $z$ in this strip. It follows  that $E_\alpha^*$ is univalent
in such a strip. This easily implies that a domain $T^*$ with the required properties
exists.

Instead of the sector 
$W_\alpha=\{z\colon|\arg z|<\alpha\pi/2\}=E_\alpha(S)$ we now consider 
the domain $W_\alpha^*=E_\alpha^*(S)$.
For $k\in\Z$ we define, analogously to the proof of Theorem~\ref{theorem1},
now $j(k)$ as the maximal integer such that
the half-line
$\{i\pi k+x\colon -\infty<x\leq \log x_{j(k)}\}$ does not intersect~$W_\alpha^*$. 

Again we apply the MacLane-Vinberg method with $c_k=(-1)^k x_{j(k)}$
and consider the comb domain $\Omega^*$ defined by~\eqref{b3}. We define $V^*$ as the component of $(E_\alpha^*)^{-1}(\Omega^*)$ that contains~$\R$.
Then $V^*\supset S$.

 Let $\varphi$ be a conformal map from the lower half-plane to $\Omega^*$ such that $\varphi(iy)\in\R$
for $y<0$.
Instead of the map $h$ defined by~\eqref{6c1} we now consider the map
\begin{equation} \label{4n}
h\colon V^*\to S,\quad 
h(w)=\log\!\left( i\varphi^{-1}(E_\alpha^*(w))\right).
\end{equation}
Again we find that there exists $\lambda\in\R$ such that~\eqref{6d} holds for every
$\delta>0$. As before we may achieve that $\lambda=0$ by a normalization and 
obtain
\begin{equation} \label{4o}
 \varphi^{-1}(E_\alpha^*(w)) \sim -i e^w
\quad\text{as}\ \re w\to\infty, \  |\im w|\leq \frac{\pi}{2}-\delta ,
\end{equation}
instead of~\eqref{6d1}.
Again we put $z=e^w$ so that $w=\log z$,
with the principal branch of the logarithm. Instead of~\eqref{6d2} and~\eqref{6e}
we now have
\begin{equation} \label{4p}
 \varphi^{-1}\!\left(\frac{z^\alpha}{((\log z)^2+c^2)^q}\right) \sim -i z
\quad\text{as}\ |z|\to\infty, \  |\arg z|\leq \frac{\pi}{2}-\delta.
\end{equation}
and
\begin{equation} \label{4q}
\varphi(-iz)\sim \frac{z^\alpha}{((\log z)^2+c^2)^q}
\sim \frac{z^\alpha}{(\log z)^{2q}}
\quad \text{as}\ |z|\to\infty,\ |\arg z|\leq \frac{\pi}{2}-\delta ,
\end{equation}
for every $\delta>0$.
The arguments used to prove~\eqref{6e4} now yield~\eqref{4l}.

The order of $g$ and the order of the function $f$ given
by $f=F\circ g$ again satisfy~\eqref{3h} and~\eqref{6e3}, but in
view of Lemma~\ref{lemma1a}, \eqref{3g} and~\eqref{4l} we also have
\begin{equation} \label{4s}
T(r,f)=\O\!\left( \frac{r^{2\alpha}}{(\log r)^{4q}}\right)
=\O\!\left( \frac{r^{\rho(f)}}{(\log r)^{4q}}\right)
\end{equation}
We defined $t$ by~\eqref{5e1}. Thus $t=d_I=d_J$ by~\eqref{4j}.
Choosing $q$ sufficiently large we deduce from Lemma~\ref{lemma2b} that~\eqref{5c}
holds.

As in~\cite{Bergweiler2012} we put $B(R)=\{z\colon |z|>R\}\cup\{\infty\}$.
Here $R\geq (16R_0)^M$ where $R_0$ is chosen such that $D(0,R_0)$ contains
all singularities of $f^{-1}$ and $R_0>|f(0)|$. Thus in our case we can 
take $R_0=2$.
Let $E_l$ denote the collection of all components $V$ of $f^{-l}(B(R))$
for which $f^k(V)\subset B(R)$ for $k=0,1,\dots,l-1$. 
It is shown in~\cite[p.~5376f]{Bergweiler2012} that $E_l$ is a cover of the set
\begin{equation} \label{4t}
\{z\in B(3R): f^k(z)\in B(3R)\ \mbox{for}\  1\leq  k \leq
l-1\}.
\end{equation}
and that
\begin{equation} \label{4u}
\sum_{V\in E_l}\left(\diam_{\chi}(V)\right)^t
\leq \frac{1}{M}\left(\frac{32}{(2R)^{1/M}24}\right)^{\! t}
\left( M (2^{1/M} 24)^{t}
\! \sum_{j=n(R)}^\infty
\left(\frac{|b_j|}{|a_j|^{1+1/M}}\right)^{\! t} \right)^{\! l}.
\end{equation}
Here $\diam_{\chi}(\cdot)$ denotes the spherical diameter.
In~\cite{Bergweiler2012} it is assumed that~\eqref{5c} holds, but the
argument also works if $t$ is defined by~\eqref{5e1}.

All the above conclusions also hold if $f$ is replaced 
by the function $f_\lambda$ defined by $f_\lambda(z)=f(\lambda z)$,
where $\lambda\in (0,1]$.
Replacing $f$ by $f_\lambda$ leaves the $b_j$ unchanged, but replaces
$a_j$ by $a_j/\lambda$.
By choosing $\lambda$ small we may thus achieve that not only~\eqref{5c} holds,
but that
\begin{equation} \label{4v}
M (2^{1/M} 24)^{t} \sum_{j=1}^\infty\left(\frac{|b_j|}{|a_j|^{1+1/M}}\right)^t < 1.
\end{equation}
We may also achieve that $f_\lambda\!\left(\overline{D(0,3R)}\right)\subset D(0,R)$ and thus 
$J(f_\lambda)\subset B(3R)$ by choosing $\lambda$ small.
This implies that $E_l$ is a cover of $J(f_\lambda)$ and that 
\begin{equation} \label{4w}
\sum_{V\in E_l}\left(\diam_{\chi}(V)\right)^t \to 0
\end{equation}
as $l\to\infty$. This yields that $\dim I(f_{\lambda})\leq \dim J(f_\lambda)\leq t=d_J=d_I$.

It remains to prove that $\dim I(f_\lambda)\geq t$.
This follows by minor modifications of the arguments in the proof of 
Theorem~\ref{theorem1}.
We omit the details.

\begin{ack}
The second author gratefully acknowledges support from Vergstiftelsen. We thank the referee for a detailed reading and valuable comments. We also thank Alexandre Eremenko for suggesting the reference [26].
\end{ack}

\bigskip

\noindent
W. Bergweiler: Mathematisches Seminar,
Christian--Albrechts--Universit\"at zu Kiel,
Heinrich--Hecht--Platz 6,
24098 Kiel, 
Germany\\
Email: bergweiler@math.uni-kiel.de

\medskip

\noindent
W. Cui: Centre for Mathematical Sciences, Lund University, Box 118, 22 100 Lund, Sweden\\
E-mail: weiwei.cui@math.lth.se

\end{document}